\documentclass[12pt,amscd]{amsart}
\usepackage[all]{xy}
\usepackage{graphicx}
\usepackage{amsmath,amsxtra,amssymb,latexsym, amscd,amsthm}
\usepackage{indentfirst}
\usepackage[mathscr]{eucal}

\newtheorem{thm}{Theorem}[section]
\newtheorem{cor}[thm]{Corollary}
\newtheorem{lem}[thm]{Lemma}
\newtheorem{prop}[thm]{Proposition}
\newtheorem*{ques}{Question}
\theoremstyle{definition}
\newtheorem{defn}[thm]{Definition}
\newtheorem{exm}[thm]{Example}
\newtheorem*{rem}{Remark}
\newtheorem*{ack}{Acknowledgements}
\numberwithin{equation}{subsection}
\DeclareMathOperator{\ext}{Ext}

\newcommand{\seq}[1]{\left<#1\right>}

\def\m{{\frak m}}

\def\a{{\bold a}}
\def\b{{\bold b}}
\def\e{{\bold e}}
\def\0{{\bold 0}}
\def\x{{\bold x}}
\def\D{{\Delta}}
\def\F{{\mathcal F}}
\def\NZQ{\Bbb}
\def\NN{{\NZQ N}}

\def\ZZ{{\NZQ Z}}

 \DeclareMathOperator{\lk}{lk}
\DeclareMathOperator{\st}{st}

\DeclareMathOperator{\Ass}{Ass}
 
\begin{document}
\title[Combinatorial characterizations of the Cohen-Macaulayness of the second power of edge ideals]
 {Combinatorial characterizations of the Cohen-Macaulayness of the second power of edge ideals}

\author{D\^o Trong Hoang}
\address{Institute of Mathematics, Box 631, Bo Ho, 10307 Hanoi, Viet Nam}
\email{dotronghoang@gmail.com}
\author{Nguy\^en C\^ong Minh}
\address{Department of Mathematics, Hanoi National University of Education, 136 Xuan Thuy, Hanoi,
Vietnam}
\email{ngcminh@gmail.com; minhnc@hnue.edu.vn}
\author{ Tr\^an Nam Trung}
\address{Institute of Mathematics, Box 631, Bo Ho, 10307 Hanoi, Viet Nam}
\email{tntrung@math.ac.vn}
\thanks{The second author partially supported by the National Foundation for Science and Technology Development (Vietnam) under grant number 101.01-2011.08. The first and third authors are partially supported by the National Foundation for Science and Technology Development (Vietnam) under grant number 101.01-2011.48.} 
\subjclass[2000]{13D45, 05C90, 05E40, 05E45}
\keywords{Symbolic powers,  edge ideal, Cohen-Macaulay, Buchsbaum, generalized Cohen-Macaulay, bipartite graphs}
\date{}
\dedicatory{}
\commby{}
% -----------------------------------------------------------
\begin{abstract}
Let $I(G)$ be the edge ideal of a simple graph $G$. In this paper, we will give sufficient and necessary combinatorial conditions of $G$ in which the second symbolic and ordinary power of its edge ideal are Cohen-Macaulay (resp. Buchsbaum, generalized Cohen-Macaulay). 

As an application of our results, we will classify all bipartite graphs in which the second (symbolic) powers are Cohen-Macaulay (resp. Buchsbaum, generalized Cohen-Macaulay). 
\end{abstract}

% -----------------------------------------------------------
\maketitle
% -----------------------------------------------------------
\section*{Introduction}
Let $S=k[x_1,x_2,\ldots,x_n]$ be a polynomial ring in $n$ variables over a field $k$. Let $\Delta$ be a simplicial complex on $ [n]=\{1,\ldots,n\}$. The Stanley-Reisner ideal of $\D$ is defined as 
$$I_\D=\big(\prod_{i\in F} x_i~|~ F\notin\D\big).$$
A squarefree monomial ideal of $S$ can be written as the Stanley-Reisner ideal of a suitable simplicial complex. Cohen-Macaulayness (resp. Buchsbaumness, generalized Cohen-Macaulayness) of these ideals have been studied by several authors (see \cite{S}, \cite{St}, \cite{BH}, \cite{V}).  Recently, Minh and Trung in \cite{MT2} and Varbaro in \cite{Va} independently proved that $S/I_{\D}^{(m)}$ is Cohen-Macaulay for all $m\in\NN$ if and only if $\D$ is a matroid, where $I_{\D}^{(m)}$ is the $m^{th}$-symbolic power of $I_\D$. The matroid is one of the important concepts of discrete mathematics and is largely applied, for example the graph theory. Later on, Terai and Trung proved that $S/I_{\D}^{(m)}$ is Cohen-Macaulay for some $m\in\NN, m\ge 3$ if and only if $\D$ is a matroid \cite{TT}. They also proved similar characterizations for the properties of being Buchsbaum or generalized Cohen-Macaulay. So, it is natural to look for a characterization of these properties for the second symbolic power of the ideal $I_{\D}$. Then, here, we consider the following question:

\medskip
\noindent
{\bf Question.}
\begin{itemize}
{\it
\item[(1)] When is $S/I_{\Delta}^{(2)}$ Cohen-Macaulay?\\
\item[(2)] When is $S/I_{\Delta}^{(2)}$ Buchsbaum?\\
\item[(3)] When is $S/I_{\Delta}^{(2)}$ generalized Cohen-Macaulay?
}
\end{itemize}

The above questions are of interest for what have been mentioned at the beginning, and here we answer them under the additional assumption that $\D$ is a flag simplicial complex. Let us consider the first one: An answer to this question is already given in \cite[Theorem 2.1]{MT2} without the flag condition of $\D$ (also see \cite{MN1} and \cite{MN2} for some special cases of the second part of this question). However, \cite[Theorem 2.1]{MT2} does not give a characterization in a combinatorial fashion, rather it involves the Cohen-Macaulayness of a family $ \mathcal{F}$ of certain subcomplexes of $\D$. So it makes perfectly sense to look for a nicer characterization. In this paper, we will provide a similar result to the second symbolic power of the edge ideal of a graph $G$. In this situation, we only have to check the Cohen-Macaulayness for a set of certain subgraphs of $G$ which, roughly speaking, is a proper subset of $ \mathcal{F}$. 

Let $G$ be a graph without isolated vertices with the vertex set $V(G)=[n]$ and the edge set $E(G)$. The squarefree monomial ideal $$I(G) = \big( x_ix_j  ~|~  ij \in E(G) \big)\subseteq S,$$
is called the edge ideal of $G$. For a graph $G$, we denote by $\D(G)$, the independence complex of $G$, which is the simplicial complex on $[n]$ whose the Stanley-Reisner ideal is $I(G)$. This means that faces of $\D(G)$ are exactly the independent sets of vertices of $G$. The first main result of this paper is:

\medskip
\noindent
{\bf Theorem ~\ref{Main1}.}
{\it
Let $I(G)$ be the edge ideal of a graph $G$. Let $\D=\Delta(G)$ be the independence complex of $G$. The following conditions are equivalent:
\begin{itemize}
\item [(i)] $S/I(G)^{(2)}$ is Cohen-Macaulay.
\item[(ii)] $\D$ is Cohen-Macaulay and $\st_{\Delta}(p) \cup \st_{\Delta}(q)$ is Cohen-Macaulay for all $pq\in E(G)$.
\item[(iii)] $G$ is a Cohen-Macaulay graph and for all edge $pq$, $G_{pq}$ is Cohen-Macaulay and $\alpha(G_{pq})=\alpha(G)-1$. When this is the case, $G$ is called \em{special Cohen-Macaulay}.
\end{itemize}
}
\medskip 
\noindent
{\bf Corollary ~\ref{Cor3}.}
{\it Let $I(G)$ be the edge ideal of a graph $G$. Then, $S/I(G)^{2}$ is Cohen-Macaulay if and only if $G$ is special Cohen-Macaulay and has no triangles (i.e. it has no subgraph which forms a triangle).}

\medskip 
For the generalized Cohen-Macaulay property, we give the following characterization.
 
\medskip 
\noindent
{\bf Corollary~\ref{Cor1}.}
{\it
$S/I_{\Delta}^{(2)}$ is generalized Cohen-Macaulay if and only if $\Delta$ and $\bigcup_{i\in V}\st_{\D}(V\setminus\{i\})$ are Buchsbaum for all subsets $V\subseteq [n]$ with $1\le |V|\le \dim \Delta +1$.
}
\medskip 

(See Proposition ~\ref{thm1.3} for a more general version of this corollary).

Independence complexes belong to the class of  contractible complexes, which was introduced by R. Ehrenborg and G. Hetyei \cite{EH}, see Definition 3.6. We will give the structure of the $\ZZ^n$-graded local cohomology modules of $S/I_{\D}^{(2)}$ when $\D$ is a contractible generalized Cohen-Macaulay complex. 

\medskip
\noindent
{\bf Theorem ~\ref{Main2}.}
{\it
 Let $\D$ be a contractible simplicial complex on $[n]$. Assume that $S/I_{\D}^{(2)}$  is generalized Cohen-Macaulay. Then 
$$H^i_{\m}(S/I_{\D}^{(2)})= [H^i_{\m}(S/I_{\D}^{(2)})]_{\bf 0} \oplus \bigoplus\limits_{u=1}^n [H^i_{\m}(S/I_{\D}^{(2)})]_{{\bf e}_u} \oplus \bigoplus_{\{u,v\}\notin \D}[H^i_{\m}(S/I_{\D}^{(2)})]_{\e_u+\e_v},$$
for all $i\in\NN$ and $i<\dim(S/I_{\D}^{(2)})$.
}
\medskip 

The following theorem characterizes graphs $G$ in which the second symbolic powers $S/I(G)^{(2)}$ are Buchsbaum.

\medskip
\noindent
{\bf Theorem ~\ref{Main3}.}
{\it
Let $I(G)$ be the edge ideal of a graph $G$. The following conditions are equivalent:
\begin{itemize}
\item [(i)] $S/I(G)^{(2)}$ is Buchsbaum.
\item[(ii)] $S/I(G)^{(2)}$ is 1-Buchsbaum.
\item[(iii)] $G$  is Cohen-Macaulay  and $S/I(G)^{(2)}$ is generalized Cohen-Macaulay.
\item[(iv)] $G$ is Cohen-Macaulay and $G_i$ is special Cohen-Macaulay for all $i\in [n]$.
\end{itemize}
}

As an application of our results, we will classify all bipartite graphs in which the second symbolic powers are Cohen-Macaulay (resp. Buchsbaum, generalized Cohen-Macaulay), see Propositions ~\ref{Ex1}, ~\ref{Ex2} and ~\ref{Ex3}.

Let us explain the organization of this paper. In Section 2, we set up some basic notations and terminologies for simplicial complexes and graphs.  Section 3 is devoted to prove a combinatorial characterization of the Cohen-Macaulayness of the second (symbolic) power of an edge ideal. In Section 4, we will characterize generalized Cohen-Macaulay monomial ideals. This characterization can be checked via linear programming methods. Afterwards, we consider the generalized Cohen-Macaulay (resp. Buchsbaum) property of the second (symbolic) power of an edge ideal and prove Theorems ~\ref{Main2} and ~\ref{Main3}.  In the last section, we will give a classification of bipartite graphs in which the second symbolic powers are Cohen-Macaulay (resp. Buchsbaum, generalized Cohen-Macaulay).

\section{Preliminary}
We refer the reader to \cite{St} for detailed information about combinatorial and algebraic background. We also will use some notation for graphs according to \cite{D}. 
 
Let $I$ be a monomial ideal and $\m = (x_1,x_2,\ldots,x_n)S$ the homogeneous maximal ideal in $S$. Also $H^i_{\m}(S/I)$ denotes the $i$th-local cohomology module of $S/I$ with respect to $\m$. A residue class ring $S/I$ is called a generalized Cohen-Macaulay (resp. Buchsbaum) ring if $H^i_{\m}(S/I)$ has finite length (resp. the canonical map $$\ext_S^i(S/\m,S/I) \to H_{\m}^{i}(S/I)$$ is surjective) for all $i<\dim(S/I)$. The ideal $I$ is called unmixed if $\dim(S/I)=\dim(S/P)$ for all $P\in\Ass(S/I)$. It should be noted that $I$ is unmixed if it is Cohen-Macaulay. 

Let $\NN$ (resp. $\ZZ$) denote the set of nonnegative integers (resp. integers). A simplicial complex $\Delta$ on $ [n]=\{1,\ldots,n\}$ is a collection of subsets of $[n]$ such that $F \in \Delta$ whenever $F \subseteq F'$ for some $F' \in \Delta$. An element $F\in\D$ which is a maximal face with respect to inclusion, is called a facet of $\D$. For each $F\in\D$, we put $\dim F = |F|-1$ and $\dim \Delta = \max \{ \dim F ~|~ F \text{ is a facet of $\D$} \}$, which is called the dimension of $\Delta$. If all facets of $\D$ have the same dimension, we say that $\D$ is pure. For a simplicial complex $\D$, we define two subcomplexes
$$\lk_{\D}F=\{G\in\D\mid G\cup F\in\D, G\cap F=\emptyset\},$$ $$\st_{\D}F=\{G\in\D\mid G\cup F\in\D\}$$
for any $F\subseteq [n]$. Moreover, if $\D$ is pure and $F\in\D$ then $\dim(\lk_{\D}F)=\dim(\D)-|F|$ and $\dim(\st_{\D}F)=\dim(\D)$. We denote by $\widetilde{H}_{j}(\D;k)$ (resp. $\widetilde{H}^{j}(\D;k)$) the reduced (co)homology group of a simplicial complex $\D$ over $k$ (cf. \cite[Section 5.3]{BH}). A simplicial complex $\D$ is called Cohen-Macaulay (resp. Buchsbaum, generalized Cohen-Macaulay) if so is $k[\D]$ over any a field $k$. It is known that $\D$ is Cohen-Macaulay (resp. Buchsbaum) if and only if $\widetilde{H}_{j}(\lk_{\D}F;k)=(0)$ for all $j<\dim(\D)-|F|$, $F\in\D$ (resp. $\widetilde{H}_{j}(\lk_{\D}F;k)=(0)$ for all $j<\dim(\D)-|F|$, $F\in\D\setminus\{\emptyset\}$ and $\D$ is pure) (see \cite[Corollary 5.3.9]{BH}, \cite[Theorem 3.2]{S}). It should be noted that any Buchsbaum ring is generalized Cohen-Macaulay and the converse is also true for Stanley-Reisner rings. In the last case, $\D$ is always pure. 

For each $0\ne m\in\NN$, the $m^{th}$-symbolic power of $I_\D$ is the ideal
$$I_\D^{(m)} =  \bigcap_{F \in \F(\D)}P_F^m,$$
where $\F(\D)$ denotes set of facets of $\D$ and $P_F=(x_j~|~j\notin F)S$.

Let $I$ be a monomial ideal in $S$. In \cite{T}, Takayama found the following combinatorial formula for $\dim_kH_\m^i(S/I)_\a$ for all $\a\in\ZZ^n$ in terms of certain complexes. For every $\a = (a_1,\ldots,a_n) \in \ZZ^n$ we set $G_\a = \{i~|~\ a_i < 0\}$ and write $\x^{\a} = \Pi_{j=1}^n x_j^{a_j}$. The degree complex $\D_\a(I)$ is the simplicial complex whose faces are sets of form 
 $F \setminus G_\a$, where $G_\a\subseteq F\subseteq [n]$, so that for every minimal generator $x^\b$ of $I$ there exists an index $i \not\in F$ with $a_i < b_i$. Let $\D$ denote the simplicial complex such that $\sqrt{I}$ is the Stanley-Reisner ideal of $\D$. For $j = 1,\ldots,n$, let $\rho_j(I)$ denote the maximum of the $i$th coordinates of all vectors $\b \in \NN^n$ such that $x^\b$ is a minimal monomial generator of $I$. 

\begin{thm}[T]\label{T}
$$\dim_kH_\m^i(S/I)_\a = 
\begin{cases}
\dim_k\widetilde H_{i-|G_\a|-1}(\D_\a(I);k) & \text{\rm, if }\ G_\a \in \D\ \text{\rm and}\\
&\ \ \  \ a_j  < \rho_j(I)\ \text{\rm for}\ j=1,\ldots,n,\\
0 & \text{\rm, else. }
\end{cases} $$
\end{thm}

The degree complex can be computed in a more simple way and this result is used later. For every $F\subseteq [n]$, let $S_F=S[x_i^{-1}~|~i\in F]$ and $P_F=(x_i~|~i \notin F)S$. Then the minimal primes of $I$ are the ideals $P_F$, $F \in \F(\D)$.
Let $I_F$ denote the $P_F$-primary component of $I$. If $I$ has no embedded components, we have
$$I = \bigcap_{F \in \F(\D)}I_F.$$ 
Let $\Gamma$ be a subcomplex of $\Delta$ with $\mathcal{F}(\Gamma)\subseteq \mathcal{F}(\Delta)$. We set 
$$L_{\Gamma}(I) = \Big\{ {\a}\in \NN^n ~|~ {\bf x}^{\a} \in \bigcap_{F\in \mathcal{F}(\Delta)\setminus \mathcal{F}(\Gamma) } I_{F} \setminus \bigcup_{G\in \mathcal{F}(\Gamma)} I_G \Big\}.$$

\begin{lem}[MT2, Lemma 1.5]\label{degcomplex}
Assume that $I$ is unmixed. For $\a\in\NN^n$, we have $\D_\a(I)$ is pure of dimension $\dim(\D)$ and 
$$\F(\D_\a) = \big\{F \in \F(\D)~|~\ x^\a \not\in I_F\big\}.$$
Moreover, $\D_{\a}(I)=\Gamma$ if and only if $\a\in L_{\Gamma}(I)$.
\end{lem}

A graph $G$ consists of a finite set $V(G)$ of vertices and a collection $E(G)$ of subsets of $V(G)$ called edges, such that every edge of $G$ is a pair $\{u,v\}$ for some $u, v \in V(G)$. Throughout this paper, we assume that a graph $G$ has no loops, that is, we are requiring $u \ne v$ for $\{u,v\} \in E(G)$. Without loss of generality, we assume that $V(G)=[n]$ for some $n\in\NN$. Two vertices $u, v$ of $G$ are adjacent if $\{u,v\}$ is an edge of $G$. A set of vertices is independent if no two of its elements are adjacent. Let $\alpha(G)$ be the maximal cardinality of an independent set of vertices of $G$. One can see that $\dim(\D(G))=\alpha(G)-1$. If $U\subseteq [n]$, then $G\setminus U$ denote the subgraph of $G$ on $V(G)\setminus U$ whose edges are precisely the edges of $G$ with both vertices in $V(G)\setminus U$. The set of adjacent vertices of a vertex $u\in V(G)$ is denoted by $N(u)$ and $\deg(x)=|N(x)|$ is called the degree of $x$. $G$ is called complete if $\deg(u)=|V(G)|-1$ for all $u\in V(G)$. For simplicity, we often write $i \in G$ (resp. $ij \in E(G)$) instead of $i \in V(G)$ (resp. $\{i,j\} \in E(G)$). A graph $G$ is called a Cohen-Macaulay graph (resp. Buchsbaum) if so is its independence complex. 
%--------------------------------------------------------------------------------------------------------------------------------------------------
\section{Cohen-Macaulayness of the second power of an edge ideal}
Throughout this section, let $G$ be a graph on $[n]$ and $I(G)$ its edge ideal in $S=k[x_1,\cdots,x_n]$. Let $\D=\Delta(G)$ be the independence complex of $G$. First, we shall prove the following lemma.

\begin{lem}[H, page 98]\label{lema2.4} Let  $\Delta, \Gamma$ be Cohen-Macaulay simplicial complexes of dimension $d$ on $[n]$. Then, the following assertions hold true:
\begin{itemize}
\item[(i)] If  $\Delta\cap \Gamma$ is Cohen-Macaulay of dimension $d$, then $\Delta\cup \Gamma$ is Cohen-Macaulay of dimension $d$.
\item[(ii)] Assume $\dim(\Delta\cap \Gamma)=d-1$. Then, $\Delta\cup \Gamma$ is Cohen-Macaulay if and only if  $\Delta\cap \Gamma$ is Cohen-Macaulay.
\end{itemize}
 \end{lem}
 \begin{proof} Note that 
\begin{eqnarray*}
 \lk_{\Delta\cup \Gamma}(F)&=& \lk_{\Delta}(F)\cup  \lk_{\Gamma}(F)\\
\lk_{\Delta\cap \Gamma}(F)&=& \lk_{\Delta}(F)\cap  \lk_{\Gamma}(F)
 \end{eqnarray*}
for all $F\subseteq [n]$. Using the Mayer-Vietoris sequence and the Reisner's criterion for Cohen-Macaulayness of a simplicial complex (see \rm{[BH, Corollary 5.3.9]}), we obtain the conclusion.
\end{proof}

For each $pq\in E(G)$, let $G_{pq}=G\setminus( N(p)\cup N(q))$. We say that $G_{pq}$ is the localization of $G$ at edge $pq$. 
Suppose $U, V\subseteq [n]$ and $U\cap V=\emptyset$. Let $\Gamma$ (resp. $\Lambda $) be a simplicial complex on $U$ (resp. $V$). Then the simplicial join of $\Gamma$ and $\Lambda$, denoted by $\Gamma * \Lambda$, is defined by $\{F\cup L~|~F\in \Gamma, L\in\Lambda\}$. It is a simplicial complex on $U\cup V$. 
Our first main result is:

\begin{thm} \label{Main1} Let $I(G)$ be the edge ideal of a graph $G$. Let $\D=\Delta(G)$ be the independence complex of $G$. The following conditions are equivalent:
\begin{itemize}
\item [(i)] $S/I(G)^{(2)}$ is Cohen-Macaulay.
\item[(ii)] $\D$ is Cohen-Macaulay and $\st_{\Delta}(p) \cup \st_{\Delta}(q)$ is Cohen-Macaulay for all $pq\in E(G)$.
\item[(iii)] $G$ is a Cohen-Macaulay graph and for all edge $pq$, $G_{pq}$ is Cohen-Macaulay and $\alpha(G_{pq})=\alpha(G)-1$. When this is the case, $G$ is called \em{special Cohen-Macaulay}.
\end{itemize}
\end{thm}
\begin{proof} (i)$\Longrightarrow$(ii):\quad By \cite[Theorem 2.1]{MT2}, $\D_{\0}(I(G)^{(2)})=\D$ and $$\Delta_{\a}(I(G)^{(2)})=\st_{\Delta}(p)\cup \st_{\Delta}(q)$$ are Cohen-Macaulay, where $\a=\e_p+\e_q $ and ${\bold e}_p$ denotes the $p$-th unit vector for all $pq\in E(G)$ if $S/I(G)^{(2)}$ is Cohen-Macaulay. 

(ii)$\Longrightarrow$(i):\quad By \cite[Theorem 2.1]{MT2} to check, that for all $V\subseteq [n]$ with $2\leq |V|\leq \dim(\D)+1$ the complex $\D_V$ is Cohen-Macaulay. Here, $\D_V$ is the simplicial complex, whose facets are the facets of $\D$ having at least $|V|-1$ vertices in $V$. 

Let $G[V]$ be the induced subgraph of $G$ on the set $V$. If $G[V]$ has an induced subgraph which is a triangle, say $\{ij,jk,ki\}$, or  a pair of disjoint edges, say $\{ij,kl\}$, then  $x_ix_jx_k\in I(G)^{(2)}$ or $x_ix_jx_kx_l\in I(G)^{(2)}$, so $\Delta_V = \emptyset$ by its definition. Therefore, we may assume that  $G[V]$  has no induced subgraph which forms a triangle or a pair of disjoint edges. We have several cases of $G[V]$ as follows.

{\bf Case 1}: $G[V]$ consists of isolated vertices. Then $V\in\D$. Without loss of  generality, let $V=\{1,\ldots,m\}$. Therefore, by its definition, 
$$\Delta_V=\bigcup_{i=1}^m \st_{\Delta}(V\setminus\{i\}),$$
where $\st_{\Delta}(V\setminus\{i\})\ne \emptyset$ for all $i$. Note that
 $$\st_{\Delta}(V\setminus\{i\})=\seq{V\setminus\{i\}}*\lk_{\Delta}(V\setminus\{i\}).$$
On the other hand, for all $i\ne j\in V$,
$$\st_{\Delta}(V\setminus\{i\})\cap \st_{\Delta}(V\setminus\{j\}) =\st_{\Delta}(V)=\seq{V}*\lk_{\Delta}(V).$$
In fact, we always have $\st_{\Delta}(V)\subseteq \st_{\Delta}(V\setminus\{i\})\cap \st_{\Delta}(V\setminus\{j\})$. Let $F\in\st_{\Delta}(V\setminus\{i\})\cap \st_{\Delta}(V\setminus\{j\})$ then $F\cup(V\setminus\{i\})\in\D, F\cup(V\setminus\{j\})\in \D$. Since $\D$ is the independence complex of $G$, $F\cup V\in\D$. Then $F\in\st_{\Delta}(V)$. Thus, $\st_{\Delta}(V\setminus\{i\})\cap \st_{\Delta}(V\setminus\{j\})$ is Cohen-Macaulay simplicial complex for all $i\ne j$ (see \rm{[BH, Exerices 5.1.21]}). 

Let $$\Gamma_t =\bigcup_{i=1}^t\st_{\Delta}({V\setminus\{i\}}),$$ for all $t\le m$. We want to show that $\Gamma_t $ is Cohen-Macaulay by induction on $t$. If $t=1$ then the assertion is true. If $m>t>1$, we have  $$\Gamma_t\cap \st_{\Delta}(V\setminus\{t+1\}) =\st_{\Delta}(V) =\seq{V}*\lk_{\Delta}(V)$$ is Cohen-Macaulay. By induction and Lemma \ref{lema2.4}(i), $\Gamma_{t+1}$ is Cohen-Macaulay. In particular, $\Delta_V=\Gamma_{m}$ is Cohen-Macaulay.

{\bf Case 2}: Let  $|V|\ge 3$ and assume that $G[V]$  consists of one edge of $G$, say $pq$, and isolated vertices otherwise. Let $W=\{\text{isolated vertices of $V$}\}$. Then $W\cup\{p\}, W\cup \{q\}\in\D$ and  $$\Delta_V = \st_{\Delta}(W\cup p) \cup \st_{\Delta}(W\cup  q) = \seq{W}*\lk_{\st_{\Delta}(p)\cup \st_{\Delta}(q)} (W).$$
Since our assumption and \rm{[BH, Exerices 5.1.21]}, $\Delta_V$ is Cohen-Macaulay.

{\bf Case 3}: Let  $|V|\ge 3$ and assume that $G[V]$ consists of one star of a vertex, say $x$, which contains at least two edges and isolated vertices otherwise. Let $W=N(x)\cup \{\text{isolated vertices}\}$, so
  $$\Delta_V = \st_{\Delta}(W)  = \seq{W}*\lk_{\Delta}(W),$$
 which is Cohen-Macaulay.

{\bf Case 4}: If  $|V|=2$ and $G[V]$  consists of one edge of $G$, say $pq$. Then $\D_V=\st_{\Delta}(p) \cup \st_{\Delta}(q)$ is Cohen-Macaulay by our assumption.

Using the results just obtained we see that $\D_V$ is Cohen-Macaulay for all $V\subseteq [n]$ with $2\leq |V|$, which implies our assertions.

 (ii)$\Longrightarrow$(iii):\quad It is easily seen that the independence complex of $G_{pq}$ is 
$$ \Delta(G_{pq})=\st_{\Delta} (p)\cap \st_{\Delta}(q),$$
for all $pq\in E(G)$. In fact, if $F\in \Delta(G_{pq})$, then $F\cup\{p\}\in\D$ and $F\cup\{q\}\in\D$ (by definition of the independence complex). Hence $\Delta(G_{pq})\subseteq\st_{\Delta} (p)\cap \st_{\Delta}(q)$. On the other hand, one can see that if $F\in\st_{\Delta} (p)\cap \st_{\Delta}(q)$ then $F\cup\{p\}\in\D$ and $F\cup\{q\}\in\D$. Therefore, $N(p)\cap F=N(q)\cap F=\emptyset$ and $F\in\D$. It implies $F\in\D(G_{pq})$.

Fix $pq\in E(G)$. Using Lemma ~\ref{lema2.4}(ii), it is enough to show that $\st_{\Delta} (p)\cap \st_{\Delta}(q)$ is a simplicial complex of  dimension $\dim(\D)-1$. We will prove it by induction on $\dim(\D)$. If $\dim(\D) = 0$ then $\st_{\D}(p)\cap\st_{\D}(q)=\emptyset$ and our conclusion is trivial.  If $\dim(\D) > 0$ then 
$\widetilde{H}_{0} (\st_{\Delta} (p)\cup \st_{\Delta}(q);k)=\widetilde{H}_{-1} (\st_{\D}(p);k)=\widetilde{H}_{-1} (\st_{\D}(q);k)= (0)$ by our assumption. If $\st_{\Delta} (p)\cap \st_{\Delta}(q) =\{\emptyset\}$ then $\st_{\Delta} (p)\cup \st_{\Delta}(q)$ is not connected, a contradiction. Therefore $\st_{\Delta} (p)\cap \st_{\Delta}(q)\ne\{\emptyset\}$. Let $x\in \st_{\Delta} (p)\cap \st_{\Delta}(q)$ and $\Gamma = \lk_{\D}(x)$. Then $\Gamma$ and  $\st_{\Gamma}(p)\cup\st_{\Gamma}(q)=\lk_{\st_{\Delta} (p)\cup \st_{\Delta}(q)}(x)$ are Cohen-Macaulay of dimension $\dim(\D)-1$. By induction on $\dim(\D)$, $\st_{\Gamma}(p)\cap\st_{\Gamma}(q)=\lk_{\st_{\Delta} (p)\cap \st_{\Delta}(q)}(x)$ is a simplicial complex of  dimension $\dim(\Gamma)-1$. Then $\st_{\Delta} (p)\cap \st_{\Delta}(q)$ is a simplicial complex of  dimension $\dim(\Gamma)=\dim(\D)-1$ as required. 

 (iii)$\Longrightarrow$(ii):\quad The proof is straightforward from Lemma ~\ref{lema2.4}.
\end{proof}

\begin{cor}\label{Cor3}
Let $I(G)$ be the edge ideal of a graph $G$. Then, $S/I(G)^{2}$ is Cohen-Macaulay if and only if $G$ is special Cohen-Macaulay and has no triangles (i.e it has no subgraph which forms a triangle).
\end{cor}
\begin{proof}
Note that $S/I(G)^{2}$ is Cohen-Macaulay if and only if $S/I(G)^{(2)}$ is Cohen-Macaulay and $I(G)^2=I(G)^{(2)}$. By Theorem ~\ref{Main1} and \cite[Lemma 5.8, Theorem 5.9]{SVV}, the assertion follows. 
\end{proof}

We end this part with a remark concerning Theorem \ref{Main1} and Corollary \ref{Cor3} as follows.

\begin{rem} In light of results of Trung and Terai ([TT]) and of Rinaldo, Terai and Yoshida ([RTY]), two interesting questions for
further study arise:

(1) Is there a characterization of the Cohen-Macaulayness of $S/I(G)^{(2)}$ in terms of the graph $G$?

(2) Does the property of $S/I(G)^{(2)}$ being Cohen-Macaulay depend on the characteristic of $k$?

Theorem \ref{Main1} hints towards an affirmative answer to the first question. Yet, it is fair to say that we are still far from a general solution to this question. On the other hand, our computations suggest that in Theorem 2.3(iii) the assumption that $G$ is Cohen-Macaulay is superfluous. Indeed, we only need the condition $G_{pq}$ is Cohen-Macaulay and $\alpha(G_{pq})=\alpha(G)-1$ for all edge $pq$. It is certainly true if $G$ is a triangle-free graph, see \cite{HT}. Therefore, in Corollary \ref{Cor3}, the Cohen-Macaulay property of $G$ can be also omitted if one add more a condition that $G$ has no triangles. 

For the second question, the answer is NO in the class of bipartite graphs (see Proposition \ref{Ex1}) and in the class of triangle-free graphs (see \cite{HT}) but the other cases remain open.
\end{rem}

%----------------------------------------------------------------------------------------------------------------------------------------------------
\section{generalized Cohen-Macaulayness of the second symbolic power of a Stanley-Reisner ideal}
 Let $I$ be  a monomial ideal in $S=k[x_1,\cdots,x_n]$. Using Theorem ~\ref{T}, Takayama gave some combinatorial characterizations of the generalized Cohen-Macaulay property for $S/I$ (see \rm{[T]}). Later on, in\cite{MT1} and \cite{MT2}, the second author and N. V. Trung succeeded in using Takayama's formula to characterize the Cohen-Macaulayness of symbolic powers of $I$ in terms of its primary decomposition. Similarly, we will give here another version of  the generalized Cohen-Macaulayness of $S/I$ which can be checked by using standard techniques of linear programming. We can now formulate our result.

\begin{prop}\label{thm1.3} Let $I$ be an unmixed monomial ideal in $S$.   Then the following conditions are equivalent:
\begin{itemize}
\item[(i)] $S/I$ is generalized Cohen-Macaulay.
\item[(ii)] $\Delta_{\a}(I)$ is Buchsbaum for all $\a\in \NN^n$ (or for all $\a \in \NN^n$ with $a_j < \rho_j(I)$, $j = 1,\ldots,n$).
\item[(iii)] $L_{\Gamma}(I)=\emptyset$ for every non-Buchsbaum  subcomplex $\Gamma$ with $\mathcal{F}(\Gamma) \subseteq \mathcal{F}(\D)$.
\end{itemize}
\end{prop}

\begin{proof} (i) $\Rightarrow$ (ii): Fix $\a\in \NN^n$ such that $\D_{\a}(I)\not=\emptyset$. Let $F\in \D_{\a}(I)\setminus\{\emptyset\}$. Put ${\b}\in \mathbb Z^n$ such that $b_i=-1$ for $i\in F$ and $b_i=a_i$ for $i\notin F$, then $\lk_{\Delta_{\bf a}(I)}(F)=\Delta_{\bf b}(I)$. By Theorem ~\ref{T} and \rm{[T, Proposition 1]}, 
$$\widetilde{H}_{i-|G_{\bf b}|-1}(\Delta_{\bf b}(I);k) = H^i_{\m}(S/I)_{\bf b} =0,$$  
for all  $i<\dim(S/I)$. By Lemma ~\ref{degcomplex},
$$\dim\lk_{\Delta_{\bf a}(I)}(F) = \dim \Delta - |F| =\dim(S/I)-|G_{\bf b}|-1.$$
Then, $\widetilde{H}_j(\lk_{\Delta_{\bf a}}(F);k)=0$ for all $j<\dim (\Delta_{\bf a}(I))-|F|$ and $F\in \Delta_{\bf a}  \setminus\{\emptyset\}$, which is the desired conclusion.

(ii)$\Rightarrow$ (i): We only need to show that $\widetilde{H}_{i-|G_{\bf a}|-1} (\Delta_{\bf a}(I);k)=0$ for all ${\a}\in \ZZ^n, i<\dim(S/I), \emptyset\ne G_{\bf a}\in \Delta$. Assume $\D_{\a}(I)\ne \emptyset$. Let ${\bf b}\in \NN^n$ with $b_i=a_i$ if $a_i\ge 0$ and $b_i=0$ else. The same proof as of Theorem 1.6 (iii)$\Rightarrow$ (i) in \cite{MT2} shows that $G_{\bf a}\in \Delta_{\bf b}(I)$ and $\Delta_{\bf a}(I)=\lk_{\Delta_{\bf b}(I)}G_{\a}$. By our assumption and \cite[Theorem 1]{T}, $\Delta_{\bf b}(I)$  is generalized Cohen-Macaulay. This implies our assertion.

(ii)$\Longleftrightarrow$ (iii): By Lemma ~\ref{degcomplex}, $\Delta_{\bf a}(I) = \Gamma$ if and only if ${\bf a}\in L_{\Gamma}(I)$. This implies our assertion.
\end{proof}

\begin{exm} Let $S=k[x_1,\ldots,x_6]$ be a polynomial ring over an arbitrary field $k$. Let 
$$I=(x_4,x_5,x_6)^a\cap(x_1,x_5,x_6)^b\cap(x_1,x_2,x_6)^c\cap(x_1,x_2,x_3)^d,$$
be an ideal in $S$, where $a, b, c, d$ are non-negative integers and not all of them are zero. Then the following conditions are equivalent:
\begin{itemize}
\item[(i)] $S/I$ is generalized Cohen-Macaulay. 
\item[(ii)] $S/I$ is Cohen-Macaulay. 
\item[(iii)] $a, b, c, d$ do not satisfy all of three systems of constraints as follows:
\begin{equation}\label{D1}
         \begin{cases}
          a - 1\geq 0\\
          c - 1\geq 0\\
          a - b + c - 2\geq 0
          \end{cases}
\end{equation}

\begin{equation}\label{D2}
         \begin{cases}
          a - 1\geq 0\\
          d - 1\geq 0\\
          a - b + d - 2\geq 0\\
          a - c + d - 2\geq 0
          \end{cases}
\end{equation}

\begin{equation}\label{D3}
         \begin{cases}
          b - 1\geq 0\\
          d - 1\geq 0\\
          b - c + d - 2\geq 0
          \end{cases}
\end{equation}
\end {itemize}
\end{exm}
\begin{proof} Let $\D(I)$ denote the simplicial complex such that $\sqrt{I}$ is the Stanley-Reisner ideal of $\D(I)$. First, we can see that $\F(\D(I))$ is contained in
$$\{\{1,2, 3\},\{2,3,4\},\{3,4,5\},\{4,5,6\}\}.$$
Using Proposition ~\ref{thm1.3} and \rm[MT2, Theorem 1.6], to check the generalized Cohen-Macaulay (resp. Cohen-Macaulay) property of $S/I$, we only prove that $L_{\Gamma}(I)=\emptyset$, where $\Gamma$ is non Buchsbaum (resp. non Cohen-Macaulay) subcomplex of $\D(I)$ with $\F(\Gamma)\subseteq\F(\D(I))$. If $|\F(\Gamma)| = 1$ or $4$, then $\Gamma$ is a simplex or
 $$\F(\Gamma)=\F(\D(I))=\{\{1,2, 3\},\{2,3,4\},\{3,4,5\},\{4,5,6\}\}.$$ 
It implies that $\Gamma$ is the Cohen-Macaulay complex. If $|\F(\Gamma)| = 2$, say that $\F(\Gamma)=\{F,G\}$. Since $\Gamma$ is non Buchsbaum (resp. non Cohen-Macaulay), we have $|F\cap G|<2$. It means that $\F(\Gamma)$ is $\{\{1,2,3\},\{3,4,5\}\}$ or $\{\{1,2,3\},\{4,5,6\}\}$ or $\{\{2,3,4\},\{4,5,6\}\}$. If $|\F(\Gamma)| = 3$ and $\Gamma$ is non Buchsbaum (resp. non Cohen-Macaulay), then $\F(\Gamma)$ is $$\{\{1,2,3\},\{2,3,4\},\{4,5,6\}\}\text{ \quad or \quad}\{\{1,2,3\},\{3,4,5\},\{4,5,6\}\}.$$
This implies that (i) is equivalent to (ii). On the other hand, if the facets of $\Gamma$ are $\{\{1,2,3\},\{3,4,5\}\}$, then $L_{\Gamma}(I)=\emptyset$ is equivalent to the fact that the system of inequalities
\begin{equation}\label{D}
         \begin{cases}
          y_4 + y_5 + y_6 < a\\
          y_1 + y_2 + y_6 < c\\
          y_1 + y_5 + y_6 \geq b\\
          y_1 + y_2 + y_3 \geq d\\
          y_1\geq 0, y_2\geq 0, y_3\geq 0, y_4\geq 0, y_5\geq 0, y_6\geq 0
          \end{cases}
\end{equation}
has no integer solution. Using the Fourier-Motzkin elimination method which is a standard technique of integer programming in \cite{Sc} (also see a detailed example in \cite[Section 2]{GH}) to our system ~\eqref{D}, we will obtain that the system ~\eqref{D} has no integer solution if and only if $a, b, c, d$ do not satisfy the system of constraints \eqref{D1}. Similarly, if the facets of $\Gamma$ are $\{\{1,2,3\},\{4,5,6\}\}$ (resp. $\{\{2,3,4\},\{4,5,6\}\}$) then $a, b, c, d$ also do not satisfy the system of constraints ~\eqref{D2} (resp.  ~\eqref{D3}). Assume that the facets of $\Gamma$ are $\{\{1,2,3\},\{2,3,4\},\{4,5,6\}\}$. Then $L_{\Gamma}(I)=\emptyset$ if $a, b, c, d$ do not satisfy the system of constraints ~\eqref{D3}. Moreover, if they do not satisfy the system of constraints ~\eqref{D1}, then $L_{\Gamma}(I)=\emptyset$ in the case the facets of $\Gamma$ are $\{\{1,2,3\},\{3,4,5\},\{4,5,6\}\}$. This completes our proof.
\end{proof}

\begin{cor}\label{cor1.6} $S/I_{\Delta}^{(2)}$ is generalized Cohen-Macaulay if and only if $\Delta$ and $\bigcup_{i\in V}\st_{\D}(V\setminus\{i\})$ are Buchsbaum for all subsets $V\subseteq [n]$ with $1\le |V|\le \dim \Delta +1$.
\end{cor}
\begin{proof} It should be noted that if $S/I_{\Delta}^{(2)}$ is generalized Cohen-Macaulay, then $\Delta$ is always Buchsbaum by \cite[Theorem 2.6]{HTT}. It implies that $\D$ is pure (i.e. $S/I_{\Delta}^{(2)}$ is unmixed). We have $\rho_j(I_{\D}^{(2)}) \le 2$ for all $j = 1,\ldots, n$. By Proposition ~\ref{thm1.3}, $I_{\D}^{(2)}$ is generalized Cohen-Macaulay if and only if $\D_{\a}(I_{\D}^{(2)})$  is Buchsbaum for all ${\a} \in  \{0, 1\}^n$. By Lemma ~\ref{degcomplex}, $\D_{\a}(I_{\D}^{(2)})=\D$ if $\a=\0$ or $\a=\e_i$ for some $1\le i\leq n$, and  $\D_{\a}(I_{\D}^{(2)}) = \bigcup_{i\in V}\st_{\D}(V\setminus\{i\})$ if ${\bf a}\in \{0, 1\}^n\setminus\{\0, \e_1, \ldots, \e_n\}$, where $V = \{i\in [n]~|~ a_i = 1\}$. Moreover, $\bigcup_{i\in V}\st_{\D}(V\setminus\{i\}) =\emptyset$ if $|V | \ge \dim\Delta + 3$, and $\bigcup_{i\in V}\st_{\D}(V\setminus\{i\})$ is Cohen-Macaulay if $|V | =\dim\Delta+2$. This is the desired conclusion by Proposition ~\ref{thm1.3}.
\end{proof}

A homogeneous ideal $I$ in $S$ (or $S/I$) is called $k$-Buchsbaum if $\m^kH^i_{\m}(S/I) = (0)$ for all $i< \dim(S/I) $ (see \cite{SV}). We always have:

\begin{center}
$k$-Buchsbaumness for some $k$ $\Rightarrow$ generalized Cohen-Macaulayness.
\end{center}

From Corollary ~\ref{cor1.6}, we will obtain some necessary conditions for 1-Buchsbaumness as follows. 
 
\begin{cor}\label{Cor1}
If $S/I_{\Delta}^{(2)}$ is 1-Buchsbaum, then $\Delta$ is Cohen-Macaulay and $\bigcup_{i\in V}\st_{\D}(V\setminus\{i\})$ is Buchsbaum for all subsets $V\subseteq [n]$ with $2\le |V|\le \dim \Delta +1$.
\end{cor}
\begin{proof} Assume $S/I_{\Delta}^{(2)}$ is 1-Buchsbaum. Then $\D$  and $\bigcup_{i\in V}\st_{\D}(V\setminus\{i\})$ are Buchsbaum for all subsets $V\subseteq [n]$ with $2\le |V|\le \dim \Delta +1$ by Corollary ~\ref{cor1.6}. By contradiction, assume that $\Delta$ is not Cohen-Macaulay then $\widetilde{H}_{i} (\D;k)\ne 0$ for some $i<\dim(\D)$. Using \cite[Lemma 2.3]{MN1}, we have the following commutative diagram:
$$\CD
H_{\m}^{i+1}(S/I_{\D}^{(2)})_{\0} @>x_1>>
H_{\m}^{i+1}(S/I_{\D}^{(2)})_{\e_1} \\
@VVV @VVV \\
\widetilde{H}^{i} (\D_{\0}(I_{\D}^{(2)});k) @>>>
\widetilde{H}^{i} (\Delta_{\e_1}(I_{\D}^{(2)});k)
\endCD $$
where the bottom map is induced by the identity $\Delta_{\0}(I_{\D}^{(2)})=\Delta_{\e_1}(I_{\D}^{(2)})=\D$ (by Lemma ~\ref{degcomplex}) and the vertical maps are isormorphism as in Theorem ~\ref{T}. Hence, $x_1H_{\m}^{i+1}(S/I_{\D}^{(2)})_{\0}\ne 0$ for some $i<\dim(\D)$, a contradiction.
\end{proof}

We note that a part of Corollary ~\ref{Cor1} is independently proved in \rm{[RTY]} by using a method similar to our method. But, we will see that the converse part of this corollary holds true for an edge ideal (see Theorem  \ref{Main3}).

In \cite{EH} the following class of simplicial complexes, which contains independence complex of graphs, is introduced.

\begin{defn}    
Let $\D$ be a simplicial complex on $[n]$. An edge $\{u,v\}\in \D$ is contractible if every face $F\in\D$ satisfying $F\cup\{u\}\in\D$ and $F\cup\{v\}\in\D$ also satisfies $F\cup\{u,v\}\in\D$. The simplicial complex $\D$ is called contractible if each edge is contractible.  
\end{defn}

\begin{thm} \label{Main2} Let $\D$ be a contractible simplicial complex on $[n]$. Assume that $S/I_{\D}^{(2)}$  is generalized Cohen-Macaulay. Then 
$$H^i_{\m}(S/I_{\D}^{(2)})= [H^i_{\m}(S/I_{\D}^{(2)})]_{\bf 0} \oplus \bigoplus\limits_{u=1}^n [H^i_{\m}(S/I_{\D}^{(2)})]_{{\bf e}_u} \oplus \bigoplus_{\{u,v\}\notin \D}[H^i_{\m}(S/I_{\D}^{(2)})]_{\e_u+\e_v},$$
for all $i\in\NN$ and $i<\dim(S/I_{\D}^{(2)})$.
\end{thm}
\begin{proof} Put $ L_i= H^i_{\m}(S/I_{\D}^{(2)})$ for all $i\in\NN$ and $i<\dim(S/I_{\D}^{(2)})$. Using \rm{[T, Proposition 1]} and our assumption, we have $[L_i]_\a=0$ for all $\a\in\ZZ^n\setminus \NN^n$. It is clear that $\rho_j(I_{\D}^{(2)})\le 2$ for all $j= 1,\ldots,n$. Using Theorem ~\ref{T}, we have
$$L_i=\bigoplus_{\a\in \{0,1\}^n}[L_i]_{\a}.$$
Fix $\a\in\{0,1\}^n$ such that ${\a}\ne {\0},  {\bf e}_1,\ldots,{\bf e}_n$. Let $V=\{i\in[n]~|~a_i=1\}$, then $|V|\ge 2$. By Lemma ~\ref{degcomplex} and the definition of $\D_V$ as of the proof of Theorem ~\ref{Main1}, we have $\Delta_{\a}(I_{\D}^{(2)})=\D_V$. From this and  Theorem ~\ref{T}, $$[L_i]_{\a}\cong\widetilde{H}_{i-1}(\Delta_{\a}(I_{\D}^{(2)});k)=\widetilde{H}_{i-1}(\D_V;k).$$ 
We have several cases as follows.

{\bf Case 1}: $V\in\D$. Without loss of  generality, let $V=\{1,\ldots,m\}$. Thus
$$\Delta_V=\bigcup_{i=1}^m \st_{\Delta}(V\setminus\{i\}),$$
where $\st_{\Delta}(V\setminus\{i\})\ne \emptyset$ for all $i$. Note that
$$\st_{\Delta}(V\setminus\{i\})=\seq{V\setminus\{i\}}*\lk_{\Delta}(V\setminus\{i\}),$$
which is a cone for all $i\in V$. Similar as Case 1 in the proof of Theorem ~\ref{Main1} and $\D$ is contractible, we have
$$\st_{\Delta}(V\setminus\{i\})\cap \st_{\Delta}(V\setminus\{j\})=\st_{\Delta}(V)=\seq{V}*\lk_{\Delta}(V),$$
which implies that $\st_{\Delta}(V\setminus\{i\})\cap \st_{\Delta}(V\setminus\{j\})$ is a cone for all $i\ne j\in V$. From the Mayer-Vietoris sequence and by induction on $t$, $\widetilde{H}_{j}(\bigcup_{i=1}^t\st_{\Delta}({V\setminus\{i\}});k)=0$ for all $j$. In particular, $\Delta_V$ is acyclic. So $[L_j]_{\a}=0$ for all $j$ by Theorem ~\ref {T}.

{\bf Case 2}: $V\notin\D$. It is enough to show that $\widetilde{H}_j(\Delta_V;k)=0$ for all $j$ if $|V|\ge 3$. If there exists $i\ne j\ne q\in V$ such that $V\setminus\{i\}, V\setminus\{j\}, V\setminus\{q\}\in\D$. Since $\D$ is contractible, $\{i,j\}, \{i,q\}, \{j,q\}\notin\D$. Then $\D_{V}=\emptyset$ (by its definition). Otherwise, we have 
$\Delta_V = \st_{\Delta}(V\setminus\{p\})$ for some $p\in V$ or 
$$\D_{V}= \st_{\Delta}(V\setminus\{p\})\cup \st_{\Delta}(V\setminus\{q\}) = \seq{V\setminus\{p,q\}}*\lk_{\st_{\Delta}(p)\cup \st_{\Delta}(q)} (V\setminus\{p,q\})$$
for some $p\ne q\in V$ or $\D_{V}=\emptyset$. This implies that $\Delta_V$ is always acyclic. Hence $[L_j]_{\a}=0$ for all $j$  as shown above. Thus, the proof is complete. 
\end{proof}

We can apply Theorem ~\ref{Main2} to get the $k$-Buchsbaum property of the generalized Cohen-Macaulay second symbolic power of edge ideals. 

\begin{cor} Let $I(G)$ be the edge ideal of a graph $G$. If $S/I(G)^{(2)}$  is generalized Cohen-Macaulay, then it is $3$-Buchsbaum.
\end{cor}

It is to be noticed that we cannot replace $3$-Buchsbaumness by $2$-Buchsbaumness.

\begin{exm} Let $G$ be a square $\{12, 23, 34, 14\}$ and $I=I(G)$ its edge ideal in $S=k[x_1,x_2,x_3,x_4]$. Then, $\F(\D(G))=\{\{1,3\},\{2,4\}\}$ and $S/I^{(2)}$ is generalized Cohen-Macaulay of dimension 2. Using \rm{[MN1, Lemma 2.3]}, we have the following commutative diagram:
$$\CD
H_{\m}^{1}(S/I^{(2)})_{\0} @>x_1x_2>>
H_{\m}^{1}(S/I^{(2)})_{\e_1+\e_2} \\
@VVV @VVV \\
\widetilde{H}^{0} (\D_{\0}(I^{(2)});k) @>>>
\widetilde{H}^{0} (\Delta_{\e_1+\e_2}(I^{(2)});k)
\endCD $$
where the bottom map is identity $\Delta_{\0}(I^{(2)})=\Delta_{\e_1+\e_2}(I^{(2)})=\D(G)$ (by Lemma ~\ref{degcomplex}) and the vertical maps are isormorphism as in Theorem ~\ref{T}. Hence, $x_1x_2H_{\m}^{1}(S/I^{(2)})_{\0}\ne 0$. It implies $\m^2.H_{\m}^{1}(S/I^{(2)})\ne 0$ as required.
\end{exm}

We also will give another characterization of a graph in which the second symbolic power is generalized Cohen-Macaulay in terms of certain its subgraphs.

Let $G$ be a graph on $[n]$. For each $i\in [n]$, let $G_{i}=G\setminus(\{i\}\cup N(i))$. 

\begin{cor}\label{Cor2} 
$S/I(G)^{(2)}$ is generalized Cohen-Macaulay if and only if $G$ is unmixed and $G_i$ is special Cohen-Macaulay for all $i\in [n]$.
\end{cor}
\begin{proof} It is well-known that $S/I(G)^{(2)}$ is generalized Cohen-Macaulay if and only if $I_{\D(G)}^{(2)}S[x^{-1}_i]$ is Cohen-Macaulay for $i = 1,\ldots, n$, and $S/I(G)^{(2)}$ is equidimensional (see \cite{CST}). It is easy to check that $\D(G_{i})=\lk_{\D(G)}(i)$ for all $i\in[n]$. And, note that $I_{\D(G)}^{(2)}S[x^{-1}_i]$ is Cohen-Macaulay if and only if  $I_{\lk_{\D(G)}(i)}^{(2)}$ is Cohen-Macaulay (see \cite[Corollary 3.5]{TT} and \cite[Theorem 2.1]{MT2}). From this, our conclusion is given by Theorem ~\ref{Main1}.
\end{proof}

With the same proof of Corollary ~\ref{Cor3}, we also get the following:

\begin{cor}\label{ordinary-gCM} $S/I(G)^{2}$ is generalized Cohen-Macaulay if and only if $G$ is unmixed, and $G_i$ is special Cohen-Macaulay and has no triangles for all $i\in [n]$.
\end{cor}

It should be noted that we always have the following implications:

\begin{center}
Cohen-Macaulayness $\Rightarrow$  Buchsbaumness $\Rightarrow$ 1-Buchsbaumness 
$\Rightarrow$ generalized Cohen-Macaulayness.
\end{center}

Now, we will prove the converse of Corollary ~\ref{Cor1} for edge ideals.

\begin{thm}\label{Main3} Let $I(G)$ be the edge ideal of a graph $G$. The following conditions are equivalent:
\begin{itemize}
\item [(i)] $S/I(G)^{(2)}$ is Buchsbaum.
\item[(ii)] $S/I(G)^{(2)}$ is 1-Buchsbaum.
\item[(iii)] $G$  is Cohen-Macaulay  and $S/I(G)^{(2)}$ is generalized Cohen-Macaulay.
\item[(iv)] $G$ is Cohen-Macaulay and $G_i$ is special Cohen-Macaulay for all $i\in [n]$.
\end{itemize}
\end{thm}

\begin{proof} One can see that (i)$\Longrightarrow$(ii) is trivial.

(ii)$\Longrightarrow$(i):\quad By Corollary ~\ref{Cor1}, $\D(G)$ is Cohen-Macaulay. On the other hand, $\D_{\a}(I(G)^{(2)})=\D(G)$ if $\a=\0$ or $\a=\e_i$ for some $1\le i\leq n$ by Lemma ~\ref{degcomplex}. Using Theorem ~\ref{T}, we have
$$[H^i_{\m}(S/I(G)^{(2)})]_{\bf 0}=  [H^i_{\m}(S/I(G)^{(2)})]_{{\bf e}_u}=(0)$$ 
for all $1\le u\leq n$ and $i<\dim(S/I(G)^{(2)})$. From this and Theorem ~\ref{Main2}, one can see that 
$$H^i_{\m}(S/I(G)^{(2)})=\bigoplus_{uv\in E(G)}[H^i_{\m}(S/I(G)^{(2)})]_{\e_u+\e_v},$$
for all $i<\dim(S/I(G)^{(2)})$. By \rm{[SV, Chapter 1, Proposition 3.10]}, $S/I(G)^{(2)}$ is Buchsbaum as required.

(ii)$\Longrightarrow$(iii):\quad The proof is straightforward from Corollary ~\ref{Cor1} and Corollary ~\ref{cor1.6}.

(iii)$\Longrightarrow$(ii):\quad The same reasoning as in (ii)$\Longrightarrow$(i) shows that 
$$H^i_{\m}(S/I(G)^{(2)}) =[H^i_{\m}(S/I(G)^{(2)})]_{2}$$
for all $i<\dim S/I(G)^{(2)}$. Hence $S/I(G)^{(2)}$ is 1-Buchsbaum.

(iii)$\Longleftrightarrow$(iv):\quad It is obvious that our assertion is given by Corollary \ref{Cor2}.

\end{proof}

With the same proof as in Theorem ~\ref{Main3} and using \cite[Chapter 1, Propositon 3.10]{SV}, we also obtain the following characterization of a graph in which the second ordinary power is Buchsbaum.

\begin{thm} Let $I(G)$ be the edge ideal of a graph $G$. The following conditions are equivalent:
\begin{itemize}
\item [(i)] $S/I(G)^{2}$ is Buchsbaum.
\item[(ii)] $S/I(G)^{2}$ is 1-Buchsbaum.
\item[(iii)] $G$  is Cohen-Macaulay  and $S/I(G)^{2}$ is generalized Cohen-Macaulay.
\item[(iv)] $G$ is Cohen-Macaulay, and $G_i$ is special Cohen-Macaulay and has no triangles for all $i\in[n]$.              
\end{itemize}
\end{thm}

With the results just mentioned, it is natural to ask a question as follows.

\begin{ques}
Let $\Delta$ be a Cohen-Macaulay complex. Assume that $\bigcup_{i\in V}\st_{\D}(V\setminus\{i\})$ is Buchsbaum for all subsets $V\subseteq [n]$ with $2\le |V|\le \dim \Delta +1$. Is $S/I_{\Delta}^{(2)}$ 1-Buchsbaum (or even Buchsbaum)?
\end{ques}

%-------------------------------------------------------------------------------------------------------------------------------------------------------
\section{Applications for bipartite graphs}
To illustrate our results, in this section, we will classify all bipartite graphs in which the second symbolic power of their edge ideals are Cohen-Macaulay (resp. Buchsbaum, generalized Cohen-Macaulay). 

Recall that a graph $G$ is bipartite if its vertex set can be partitioned into two subsets $X$ and $Y$ so that every edge has one vertex in $X$ and one vertex in $Y$; such a partition $(X,Y)$ is called a bipartition of the graph, and $X, Y$ its parts. If every vertex in $X$ is joined to every vertex in $Y$ then $G$ is called a complete bipartite graph, which is denoted by $K_{|X|,|Y|}$. In \rm{[HH]}, they gave a classification of Cohen-Macaulay bipartite graphs. For later use, we also quote this result.

\begin{thm}[\cite{HeH}, Theorem 3.4] \label{HH} Let $G$ be a bipartite graph whose bipartition is $(V_1, V_2)$. Then, $G$ is Cohen-Macaulay if and only if  $|V_1|=|V_2|$ and the vertices $V_1=\{x_1,\ldots,x_n\}$ and $V_2=\{y_1,\ldots,y_n\}$ can be labeled such that: 
\begin{enumerate}
\item $x_i y_i$ is an edge of $G$ for all $1\le  i \le n$.
\item If $x_i y_j$ is an edge of $G$, then $i \le j$.
\item   If $x_i y_j$ and $x_j y_k$ are two edges of $G$ with $i < j < k$, then $x_i y_k$ is also an edge of $G$.
\end{enumerate}
\end{thm} 

\begin{prop}\label{Ex1}
Let $I(G)$ be the edge ideal of a bipartite graph $G$. Then, $S/I(G)^{(2)}$ is Cohen-Macaulay if and only if $G$ is a disjoint union of edges (i.e. $I(G)$ is a complete intersection). 
\end{prop}
\begin{proof} Assume that $S/I(G)^{(2)}$ is Cohen-Macaulay and $G$ is not a disjoint union of edges. Then, there exists an edge $x_iy_j\in E(G)$ for ($i < j$) by Theorem ~\ref{HH}. Since $G$  is special Cohen-Macaulay by Theorem ~\ref{Main1}, $G_{x_iy_j}$ is Cohen-Macaulay. One can check that $G_{x_iy_j}$ is also a bipartite graph whose bipartition is $(V_1\setminus N(y_j), V_2\setminus N(x_i))$ and $y_i,y_j\notin V_2\setminus N(x_i)$. By Theorem ~\ref{HH}, we have 
$$\alpha(G_{x_iy_j})=|V_2\setminus N(x_i)|\le |V_2|-2=\alpha(G)-2,$$ which is a contradiction.

Note that if $G$ is the disjoint union of edges, then $G$ is Cohen-Macaulay, and $G_{x_iy_i}$ is also the disjoint union of edges and $\alpha(G_{x_iy_i})=\alpha(G)-1$ for any edge $x_iy_i\in E(G)$. Then, it is clear that the converse part is given by Theorem ~\ref{Main1}. 
\end{proof}

\begin{prop}\label{Ex2} Let $G$ be a bipartite graph. Then, $S/I(G)^{(2)}$ is Buchsbaum if and only if $G$  is a path of length 3 or a disjoint union of edges.
\end{prop}
\begin{proof} Assume that $S/I(G)^{(2)}$ is Buchsbaum but not Cohen-Macaulay. By Corollary \ref{Cor1}, $G$ is Cohen-Macaulay. Then, $G$ is a graph as in Theorem \ref{HH}. If $n\le 2$, combining our assumption and Proposition ~\ref{Ex1} gives $G$ must be the path of length 3 and $n=2$. If $n>2$, by Theorem \ref{Main3}, $H=G_{y_1}$ is special Cohen-Macaulay. Note that $H$ is also a bipartite graph whose bipartition is $(V_1\setminus N(y_1), V_2\setminus N(x_1))$. Arguing as in the proof of Proposition ~\ref{Ex1}, we can see that $H$ is the disjoint union of edges. Hence, $N(x_i)=\{y_i\}$ for all $2\le i\le n$. 

Next observe that if $S/I(G)^{(2)}$ is not Cohen-Macaulay then $G$ is not the disjoint union of edges.  Therefore, there exists an edge $x_1y_j\in E(G)$ for some $j >1$. Similarly, $K=G_{x_n}$ is also the disjoint union of edges. It implies $x_1y_n\in E(G)$. By the same way, we have $L=G_{x_2}$ is also the disjoint union of edges, which is a contradiction since $x_1y_n\in L$. 
\end{proof}

\begin{prop}\label{Ex3} Let $G$ be a bipartite graph. Then, $S/I(G)^{(2)}$ is generalized Cohen-Macaulay if and only if $G$ is a complete bipartite graph $K_{n,n}$ for some $n\ge 2$ or a path of length 3 or a disjoint union of edges.
\end{prop}
\begin{proof} If $G$ is a complete bipartite graph $K_{n,n}$ for $n\ge 2$, then $S/I(G)^{(2)}$  is generalized Cohen-Macaulay but not Buchsbaum by Corollary ~\ref{Cor2} and Proposition ~\ref{Ex2}.

Now, we will prove the converse part. From $S/I(G)$ is unmixed (see \cite[Theorem 2.6]{HTT}), we may assume $(\{x_1,\ldots,x_n\},\{y_1,\ldots,y_n\})$ is the bipartition of $G$.

Assume $G$  is not connected. Set $$G=\bigcup_{i=1}^tW_i,$$
where $W_i$ is a connected component of $G$ for $1\le i\le t$. Fix $F_i\in\mathcal F(\D(W_i))$ for each $i=1,\ldots,t$. Since our assumption, $I(G)^{(2)}S[x_u^{-1}~|~u\in\cup_{i\ne j, 1\le j\le n}F_j]$ is Cohen-Macaulay. By \cite[Corollary 3.5]{TT}, $I_{\lk_{\D(G)}(\cup_{i\ne j, 1\le j\le n}F_j)}^{(2)}=I_{\D(W_i)}^{(2)}=I(W_i)^{(2)}$ is Cohen-Macaulay. Then $W_i$ consists of one edge by Proposition ~\ref{Ex1}. It implies that $G$ is the disjoint union of edges. Hence, $S/I(G)^{(2)}$ is Cohen-Macaulay.

Assume $G$ is connected. Let $x$ be the vertex of minimal degree of $G$ and $xy\in E(G)$. By Corollary ~\ref{Cor2}, $S/I(G_{x})^{(2)}$ is Cohen-Macaulay. It implies that $G_{x}$ is the disjoint union of edges or isolated vertices. Assume $G_x=\{x_{i_1}y_{i_1},\ldots,x_{i_{r}}y_{i_{r}}\}$ (i.e. $xy_{i_j}\notin G$ for all $1\le j\le r$ and $xz\in E(G)$ for all $z\notin\{y_{i_1},\ldots,y_{i_r}\}$). Since $G$ is connected, $yx_{i_j}\in E(G)$ for all  $j=1,\ldots,r$. 

{\bf Case 1}: $\deg (x)=1$. Then, $r=n-1$. If $r > 1$, then $S/I(G_{y_{i_r}})^{(2)}$ is Cohen-Macaulay (by Corollary ~\ref{Cor2}). So $G_{y_{i_r}}$ is the disjoint union of edges, which is a contradiction. If $r=1$, then $G$ is $\{xy, yx_{i_1}, x_{i_1}y_{i_1}\}$ which is the path of length 3. Hence, $S/I(G)^{(2)}$ is Buchsbaum (see Proposition ~\ref{Ex2}).

{\bf Case 2}: $\deg (x)>1$. If $r>0$, similarly, $G_{y_{i_r}}$ is also the disjoint union of edges. Then $\deg (y_{i_r})=1<\deg(x)$, which is a contradiction for choicing $x$. If $r=0$, then $xy_{i}\in E(G)$ for all $1\le i \le n$. From the minimality of $\deg (x)$, it follows that $G$ must be the complete bipartite graph, which completes the proof.

\end{proof}
\begin{ack}  We are grateful to N. V. Trung and L. T. Hoa for many suggestions and discussions on the results of this paper. We also thank to the two anonymous referees of this paper for their very useful corrections and suggestions.
\end{ack}

\end{document}